\documentclass[12pt,a4paper,reqno]{amsart}

\usepackage{amscd,amssymb,amsthm}
\usepackage{graphicx}
\theoremstyle{plain}
\newtheorem{theorem}{Theorem}[section]
\newtheorem{corollary}[theorem]{Corollary}
\newtheorem{lemma}[theorem]{Lemma}

\theoremstyle{definition}

\theoremstyle{remark}
\newtheorem{remark}{Remark}[section]

\numberwithin{equation}{section}

\numberwithin{table}{section}

\numberwithin{figure}{section}

\usepackage{enumerate}

\newcommand{\real}{\mathbb{R}}

\title[An Optimal Two Parameter Bounds for the Identric Mean]
{An Optimal Two Parameter Bounds\\
 for the Identric Mean}
\author{Omran Kouba}
\address{Department of Mathematics \\
Higher Institute for Applied Sciences and Technology\\
P.O. Box 31983, Damascus, Syria.}
\email{omran\_kouba@hiast.edu.sy}


\begin{document}

\begin{abstract}
In this note we obtain sharp bounds for the identric mean in terms
of a two parameter family of means. Our results 
generalize and extend recent bounds due to Y. M. Chu \& al. (2011), and
to M.-K. Wang \&  al. (2012).
\end{abstract}

\maketitle

\section{Introduction }\label{sec1}

Given two distinct positive real numbers $a$ and $b$, we recall that the arithmetic mean $A(a,b)$, the geometric mean
$G(a,b)$, the harmonic mean $H(a,b)$, and the identric mean $I(a,b)$,  are respectively defined by
\[
A(a,b)=\frac{a+b}{2},~~ G(a,b)=\sqrt{ab},~~H(a,b)=\frac{2ab}{a+b},
~~ I(a,b)=\frac{1}{e}\left(\frac{a^a}{b^b}\right)^{1/(a-b)}.
\]

\medskip
Inequalities relating means in two arguments have attracted and continue to attract the attention of mathematicians.
Many recent papers  were concerned in comparing these  means. 

\medskip

For instance, H. Alzer and S. Qui
considered in \cite{alqu} the following inequality relating the identric, geometric and arithmetic means :
\[
\alpha A(a,b)+(1-\alpha)G(a,b) < I(a,b)<\beta A(a,b)+(1-\beta)G(a,b),
\]
they proved that it holds, for every distinct positive numbers $a$ and $b$,  if and only if $\alpha \leq 2/3$ and $\beta\geq 2/e$.

\medskip

This was later complemented by T. Trif \cite{trif} who proved that, for $p\geq2$ and every
distinct positive numbers $a$ and $b$, we have
\[
\alpha A^p(a,b)+(1-\alpha)G^p(a,b) < I^p(a,b)<\beta A^p(a,b)+(1-\beta)G^p(a,b),
\]
if and only if $\alpha \leq (2/e)^p$ and $\beta\geq 2/3$.

\medskip

In another direction we proved in \cite{kou} that the inequality
\[ 
I^p(a,b)<\frac{2}{3}A^p(a,b)+\frac{1}{3}G^p(a,b) 
\]
holds true for every distinct positive numbers $a$ and $b$, if and only if $p\geq\ln\left(\frac{3}{2}\right)/\ln\left(\frac{e}{2}\right)\approx 1.3214$, and that the reverse inequality holds true for every distinct positive numbers $a$ and $b$, if and only if $p\leq6/5= 1.2$.

\medskip

In this paper we consider the two parameter family  of means $Q_{t,s}(a,b)$, defined for $s\geq1$ and $t\in[0,1/2]$, 
by
\begin{equation}\label{eq:int01}
Q_{t,s}(a,b)=G^s(t a+(1-t)b,tb+(1-t)a)A^{1-s}(a,b).
\end{equation}

\bigskip\goodbreak
Similar means were previously considered  by several authors. For instance 
\[Q_{t,2}(a,b)=H(t a+(1-t)b,tb+(1-t)a)\]
was considered in by  Y.-M. Chu, M.-K. Wang and Z.-K. Wang in \cite{chu1} where it was compared to the identric mean. 
The same authors compared also
\[
Q_{t,1}(a,b)=G(t a+(1-t)b,tb+(1-t)a)
\] 
to the identric mean in their recent work \cite{chu2}.

\medskip

We will see later that, for distinct positive real numbers $a$ and $b$, the function $t\mapsto Q_{t,s}(a,b)$ is 
continuous and increasing. Moreover, for  $s\geq1$ and every distinct positive numbers $a$ and $b$, we have 
\[
Q_{0,s}(a,b)\leq Q_{0,1}(a,b)=G(a,b)<I(a,b)<A(a,b)=Q_{1/2,s}(a,b).
\]
Therefore, it is natural to consider, for $s\geq1$, the sets
\begin{align*}
{\mathcal L}_s&=\left\{t\in[0,1/2]:\hbox{ for all positive $a,b$ with $a\ne b$,}~ Q_{t,s}(a,b)<I(a,b)\right\},\\
{\mathcal U}_s&=\left\{t\in[0,1/2]:\hbox{ for all positive $a,b$ with $a\ne b$,}~ I(a,b)<Q_{t,s}(a,b)\right\}.
\end{align*}

\medskip

 Using
the fact that $t\mapsto Q_{t,s}(a,b)$ is  increasing, we see that ${\mathcal L}_s$ and ${\mathcal U}_s$ are intervals.

\medskip

In this work, (see  Theorem~\ref{th31}), we will determine in terms of $s\geq1$, the values $p_s\in(0,1/2)$
and $q_s\in(0,1/2)$ such that ${\mathcal L}_s=[0,p_s]$ and ${\mathcal U}_s=[q_s,1/2]$. These results extend those of  Y.-M. Chu \& al. \cite{chu1} and  M.-K. Wang \&  al.  \cite{chu2}, with
simpler and unified proofs.

\bigskip

\section{ Preliminaries }\label{sec2}

The following lemmas pave the way to the main theorem. 
In the next Lemma~\ref{lm22} we study a family of functions, using simple methods from classical analysis.

\bigskip
\begin{lemma}\label{lm22}
For $s\geq 1$ and $u\in[0,1]$, we consider the real function $f_{u,s}$ defined on $[0,1)$ by
\begin{equation}\label{eq:e21}
f_{u,s}(x)=1-\frac{1}{2x}\ln\left(\frac{1+x}{1-x}\right)-\frac{1}{2}\ln(1-x^2)+\frac{s}{2}\ln(1-ux^2).
\end{equation}
\begin{enumerate}[\upshape(a)]
\item  The necessary and sufficient condition to have $f_{u,s}(x)>0$ for $x\in(0,1)$, is that $3su\leq 1$.
\item  The necessary and sufficient condition to have $f_{u,s}(x)<0$ for $x\in(0,1)$, is that $u+(2/e)^{2/s}\geq 1$.
\end{enumerate}
\end{lemma}

\smallskip
\begin{proof}
We consider only the case $u\in(0,1]$, since $f_{0,s}$ is independent of $s$ and positive on $(0,1)$. It is straightforward to see that $f_{u,s}^\prime(x)=h_{u,s}(x)/x^2$ where
\[
h_{u,s}(x)=-x+\frac{1}{2}\ln\left(\frac{1+x}{1-x}\right)-\frac{s u x^3}{1-u x^2}
\]
and that 
\[
h_{u,s}^\prime(x)=\frac{x^2}{(1-x^2)(1-ux^2)^2}\,T_{u,s}(x^2)
\]
where $T_{u,s}$ is the trinomial defined by
\[
T_{u,s}(X)=(1-s)u^2\,X^2-(2-3 s-s u)u\,X+(1-3 s u).
\]

\bigskip\goodbreak

\noindent Noting that $T_{u,s}(1)=(1-u)^2\geq0$ and $T_{u,s}(0)=1-3su$, we see that we have two cases:
\medskip
\begin{itemize} 
\item First, $T_{u,s}(0)\geq0$, or equivalently $3su\leq 1$. Again, we distinguish two cases :
\medskip
\begin{itemize}
\item If $s=1$, then clearly the zero of $T_{u,1}$ does not belong to $(0,1)$ and $T_{u,s}$ has a positive sign on $(0,1)$.
\medskip
\item If $s>1$, then the coefficient of $X^2$ in $T_{u,s}$ is negative, and the fact that both $T_{u,s}(0)$ and
$T_{u,s}(1)$ are nonnegative, implies that $z_0\leq 0<1\leq z_1$ where $z_0$ and $z_1$ are the zeros of $T_{u,s}$.
Hence, $T_{u,s}$ has also a positive sign on $(0,1)$ in this case. 
\end{itemize}

\medskip
\noindent It follows that in this case $h_{u,s}$ is increasing on $[0,1)$. But $h_{u,s}(0)=0$, so  $h_{u,s}$ is positive on
$(0,1)$. This implies that $f_{u,s}$ is increasing on $(0,1)$. Finally, the fact that $\lim_{x\to 0^+}f_{u,s}(x)=0$ implies that
$f_{u,s}(x)>0$ for every $x\in(0,1)$ in this case.

\medskip
\item  Second,  $T_{u,s}(0)<0$, or equivalently $3su> 1$. This means that $T_{u,s}$ has a {\it unique} zero $z_0$ in the interval $(0,1]$, (because $\deg(T_{u,s})\leq 2$ .) 
\medskip
\begin{itemize}
\item 
If $u=1$, then $z_0=1$ and $h_{1,s}$ is decreasing on $[0,1]$. But $h_{1,s}(0)=0$, so $h_{1,s}$
is negative on $(0,1)$. This implies that $f_{1,s}$ is decreasing on $(0,1)$.
Finally, we have $\lim_{x\to 0^+}f_{1,s}(x)=0$ and consequently $f_{1,s}(x)<0$ for every $x\in(0,1)$.
\medskip
\item 
If $u<1$, then $z_0\in(0,1)$. So $h_{u,s}$ is decreasing on $[0,z_0]$ and increasing on $[z_0,1]$.
But $h_{u,s}(0)=0$ so $h_{u,s}(z_0)<0$. On the other hand
$\lim_{x\to 1^-}h_{u,s}(x)=+\infty$. So there exists a unique real number $y_0\in(z_0,1)$ such that $ h_{u,s}(y_0)=0$.
Thus $ h_{u,s}(x)<0$ for $x\in(0,y_0)$ and $ h_{u,s}(x)>0$ for $x\in(y_0,1)$. This implies that $f_{u,s}$ is decreasing on $(0,y_0)$ and increasing on $(y_0,1)$. Finally we have $\lim_{x\to 0^+}f_{u,s}(x)=0$ and $\lim_{x\to 1^-}f_{u,s}(x)=\ln\left(e(1-u)^{s/2}/2\right)$.
\end{itemize}
\medskip
\noindent This shows that the necessary and sufficient condition for $f_{u,s}$ to be negative on $(0,1)$ is that 
$u=1$ or  $u<1$ and $\ln\left(e(1-u)^{s/2}/2\right)\leq0$ which is equivalent to the condition $1\leq u+(2/e)^{2/s}$.
\end{itemize}
\medskip
\noindent This achieves the proof of Lemma~\ref{lm22}.
\end{proof}

\bigskip

Next we introduce the set ${\mathcal D}$ defined as follows :
\begin{equation*}
{\mathcal D}=\left\{(a,b)\in\real^2:a>b>0\right\}.
\end{equation*}

\noindent It is sufficient to consider couples $(a,b)$ from ${\mathcal D}$, since 
the considered means are symmetric functions of their arguments.
The next Lemma~\ref{lm21} explains why the family of functions studied in  Lemma~\ref{lm22} is important to our 
study. 

\begin{lemma}\label{lm21} Consider $(a,b)\in {\mathcal D}$ and  let $v=\frac{a-b}{a+b}$.

\begin{enumerate}[\upshape(a)]
\item For  $s\geq1$ and $t\in[0,1/2]$, we have
\[
\ln\left(\frac{Q_{t,s}(a,b)}{A(a,b)}\right)=\frac{s}{2}\ln\left(1-(1-2t)^2 v^2\right).
\]
\item Also, for the identric mean we have
\[
\ln\left(\frac{I(a,b)}{A(a,b)}\right)=-1+\frac{1}{2}\ln(1-v^2)+\frac{1}{2v}\ln\left(\frac{1+v}{1-v}\right).
\]
\end{enumerate}
\end{lemma}
\begin{proof} 
Indeed, (a) follows from the simple fact that 
\[
G(ta+(1-t)b,tb+(1-t)a)=A(a,b)\sqrt{1-(1-2t)^2\left(\frac{a-b}{a+b}\right)^2}
\]

\noindent To see (b) we note that
\begin{align*}
\frac{I(a,b)}{A(a,b)}&=\frac{1}{e}\frac{2}{a+b}a^{a/(a-b)}b^{-b/(a-b)}
=\frac{1}{e}\left(\frac{2a}{a+b}\right)^{\frac{a}{a-b}}\left(\frac{2b}{a+b}\right)^{\frac{-b}{a-b}}\\
&=\frac{1}{e}\left(1+\frac{a-b}{a+b}\right)^{\frac{1}{2}+\frac{a+b}{2(a-b)}}
\left(1-\frac{a-b}{a+b}\right)^{\frac{1}{2}-\frac{a+b}{2(a-b)}}\\
&=\frac{1}{e}\big(1+v \big)^{\frac{1+v}{2v}}
\big(1-v\big)^{\frac{v-1}{2v}}.
\end{align*}
\noindent This concludes the proof of Lemma~\ref{lm22}.
\end{proof}

\bigskip

\begin{remark}\label{rm23}
In particular, it follows from Lemma~\ref{lm22} (a), that the function
$t\mapsto Q_{t,s}(a,b) $ is continuous and increasing as announced in the introduction.
\end{remark}

\begin{remark}\label{rm24}
Combining (a) and (b) from Lemma~\ref{lm22}, we see immediately that
 if $f_{u,s}$ is the function defined in Lemma ~\ref{lm22} then, for every $(a,b)\in{\mathcal D}$ we have
\[
\ln\left(\frac{Q_{t,s}(a,b)}{I(a,b)}\right)=f_{(1-2t)^2,s}\left(\frac{a-b}{a+b}\right),
\]
and this explains the importance of the family of functions studied in  Lemma~\ref{lm22} to our 
study. 
\end{remark}

\bigskip
\section{The Main Theorem }\label{sec3}

\medskip

\begin{theorem}\label{th31}
Let $s$ be a real number such that $s\geq1$, and define the sets
\begin{align*}
{\mathcal L}_s&=\left\{t\in [0,1/2]  : \forall\,(a,b)\in {\mathcal D},~ Q_{t,s}(a,b)<I(a,b)\right\},\\
{\mathcal U}_s&=\left\{t\in [0,1/2] : \forall\,(a,b)\in {\mathcal D},~ I(a,b)<Q_{t,s}(a,b)\right\}.
\end{align*}
Then 
\[
{\mathcal L}_s=\left[0,\frac{1}{2}-\frac{1}{2}\sqrt{1-\left(\frac{2}{e}\right)^{2/s}}\right]
\quad\hbox{and}\quad
{\mathcal U}_s=\left[\frac{1}{2}-\frac{1}{2\sqrt{3s}},\frac{1}{2}\right].
\]
\end{theorem}

\begin{proof} First note that
\[
\left\{\frac{a-b}{a+b}:(a,b)\in{\mathcal D}\right\}=\big(0,1\big).
\]
\noindent So, using Remark~\ref{rm24} we see that $t\in{\mathcal L}_s$ if and only if 
$f_{(1-2t)^2,s}(x)<0$ for every $x\in(0,1)$. Using Lemma~\ref{lm22} we see that this is equivalent to
$(1-2t)^2+(2/e)^{2/s}\geq 1$ or $(1-\sqrt{1-(2/e)^{2/s}})/2\geq t$. This proves that
\[
{\mathcal L}_s=\left[0,\frac{1}{2}-\frac{1}{2}\sqrt{1-\left(\frac{2}{e}\right)^{2/s}}\right].
\] 

\medskip

\noindent Similarly using Remark~\ref{rm24} we see that $t\in{\mathcal U}_s$ if and only if 
$f_{(1-2t)^2,s}(x)>0$ for every $x\in(0,1)$. Using Lemma~\ref{lm22}  again we see that this is equivalent to
$3s(1-2t)^2\leq 1$ or $(1-1/\sqrt{3s})/2\leq t$. This proves that
\[
{\mathcal U}_s=\left[\frac{1}{2}-\frac{1}{2\sqrt{3s}},\frac{1}{2}\right].
\] 

\noindent The proof of Theorem~\ref{th31} is complete.
\end{proof}
\bigskip

The following two corollaries correspond to the particular cases $s=2$ and $s=1$.
They give  the bounds obtained in \cite{chu1} and \cite{chu2}. 

\bigskip
\begin{corollary}[see \cite{chu1}]\label{cor33}
The necessary and sufficient condition on $p,q$ from $[0,1/2]$
 to have 
\[
H(pa+(1-p)b,pb+(1-p)a)<I(a,b)<H(qa+(1-q)b,qb+(1-q)a)
\]
 for every distinct positive numbers $a$ and $b$, is that 
\[
p\leq \frac{1-\sqrt{1-2/e}}{2}\qquad\hbox{ and }\qquad
q\geq \frac{6-\sqrt{6}}{12}.
\]
\end{corollary}

\bigskip
\begin{corollary}[see \cite{chu2}]\label{cor34}
The necessary and sufficient condition on $p,q$ from $[0,1/2]$
 to have 
\[
G(pa+(1-p)b,pb+(1-p)a)<I(a,b)<G(qa+(1-q)b,qb+(1-q)a)
\]
for every distinct positive numbers $a$ and $b$, is that 
\[
p\leq \frac{1-\sqrt{1-4/e^2}}{2}\qquad\hbox{ and }\qquad
q\geq \frac{3-\sqrt{3}}{6}.
\]
\end{corollary}

In the next corollary, the lower bound is an inequality due to H.-J. Seiffert \cite{sei}, and can be also found in \cite{san}.
While the upper bound is new and to be compared with the results of J. S\'andor and T. Trif in \cite{san}.

\bigskip
\begin{corollary}\label{cor35}
For every positive numbers $a$ and $b$, we have 
\[
\exp\left(\frac{1}{6}\left(\frac{a-b}{a+b}\right)^2\right)
\leq\frac{A(a,b)}{I(a,b)}\leq\exp\left(\left(\ln\frac{e}{2}\right)\left(\frac{a-b}{a+b}\right)^2\right)
\]
\end{corollary}
\begin{proof} Indeed, for $s\geq 1$ let
\[
p_s=\frac{1}{2}-\frac{1}{2}\sqrt{1-\left(\frac{2}{e}\right)^{2/s}}
\qquad\hbox{and}\qquad
q_s=\frac{1}{2}-\frac{1}{2\sqrt{3s}}.
\]

Using Theorem~\ref{th31}, for every $(a,b)\in{\mathcal D}$, we have
\[
Q_{p_s,s}(a,b)<I(a,b)<Q_{q_s,s}(a,b).
\]
This can be written as follows
\[
\frac{A(a,b)}{Q_{q_s,s}(a,b)}<\frac{A(a,b)}{I(a,b)}<\frac{A(a,b)}{Q_{p_s,s}(a,b)},
\]
and using Lemma~\ref{lm21}, it is equivalent to
\[
\left(1-\frac{v^2}{3s}\right)^{-s/2}<\frac{A(a,b)}{I(a,b)}
<\left(1-\left(1-\bigg(\frac{2}{e}\bigg)^{2/s}\right)v^2\right)^{-s/2}
\]
where $v=(a-b)/(a+b)$. Now letting $s$ tend to $+\infty$ we obtain
\[
e^{v^2/6}\leq \frac{A(a,b)}{I(a,b)}\leq e^{\left(\ln\frac{e}{2}\right)v^2},
\]
which is the conclusion of  Corollary~\ref{cor35}.
\end{proof}

\bigskip
In fact, because of the ``limit argument'' in the proof of Corollary~\ref{cor35}, we lost the strict inequalities for
distinct positive real arguments. But,
studying the family of functions $(g_t)_{t\in(0,+\infty)}$ defined by
\[
g_{t}(x)=1-\frac{1}{2x}\ln\left(\frac{1+x}{1-x}\right)-\frac{1}{2}\ln(1-x^2)-t x^2,
\]
using similar arguments to those used in Lemma~\ref{lm22}, we can prove the following exact version of
Corollary~\ref{cor35}, which extends the results of Seiffert \cite{sei} and those of S\'andor and Trif \cite{san}.

\bigskip

\begin{theorem} The necessary and sufficient condition on $p,q$ from $(0,+\infty)$ to have
\[
\forall\,(a,b)\in{\mathcal D},\qquad \exp\left(p\left(\frac{a-b}{a+b}\right)^2\right)<\frac{A(a,b)}{I(a,b)}
<\exp\left(q\left(\frac{a-b}{a+b}\right)^2\right)
\] 
is that $p\leq \frac{1}{6}$ and $q\geq \ln(\frac{e}{2})$.
\end{theorem}

\end{document}